\documentclass[11pt,a4paper]{amsart}

\usepackage{extsizes}
\usepackage{enumerate}
\usepackage{amsmath}
\usepackage{amsthm}
\usepackage{amssymb}
\usepackage{amsbsy}
\usepackage{amsfonts}
\usepackage{amstext}
\usepackage{amscd}
\usepackage{tikz}
\usepackage{tikz-network}
\usepackage{graphicx}
\usetikzlibrary{shapes.geometric}
\numberwithin{equation}{section}
\theoremstyle{plain}
\newtheorem{thm}{Theorem}[section]
\newtheorem{prop}[thm]{Proposition}

\newtheorem{lem}[thm]{Lemma}
\theoremstyle{definition}
\newtheorem{exa}[thm]{Example}

\newtheorem{prob}[thm]{Problem}
\newtheorem{rem}[thm]{Remark}
\newtheorem{defi}[thm]{Definition}

\title{Energy of a graph and Randi\'c index of subgraphs}

\author{Gerardo Arizmendi and   Diego Huerta}

\date{ \today}

\begin{document}

\maketitle


\begin{abstract}
We give a new inequality between the energy of a graph and a weighted sum over the edges of the graph.  Using this inequality we prove that  $\mathcal{E}(G)\geq 2R(H)$, where $ \mathcal{E}(G)$ is the energy of a graph $G$ and $R(H)$ is the Randi\'c index of any subgraph of $G$ (not necessarily induced). In particular, this generalizes well-known inequalities $\mathcal{E}(G)\geq 2R(G)$  and  $\mathcal{E}(G)\geq 2\mu(G)$ where $\mu(G)$ is the matching number. We give other inequalities as applications to this result.

\end{abstract}



\section{Introduction}

In this paper we consider simple undirected graphs $G=(V,E)$. We denote the cardinality of $V$ and $E$ by $n$ and $m$ respectively. We also denote the  adjacency matrix of $G$ by  $A(G)$, and its eigenvalues by $\lambda_1\leq\dots\leq\lambda_n$. The 
energy of $G$ is defined as

$$\mathcal E(G)=\sum_{i=1}^n |\lambda_i|.$$

The energy of a graph $\mathcal E(G)$ has been widely studied, see for example  \cite{Gut, Gut2, GFPR,GR,LSG}.
A graph $G$ is said to be 
hypoenergetic if $\mathcal E(G) < n$,
if $\mathcal E(G) \geq n$
then $G$ is said to be non-hypoenergetic
(see \cite{hypo}).

Considering $|A(G)|$ as the only semipositive matrix such that $A(G)^2=|A(G)|^2$ then we have that

$$\mathcal E(G)=Tr(|A(G)|)=\sum_{i=1}^n |A(G)|_{ii}.$$

This observation leads to the following important definition of the energy of a vertex, introduced and developed in \cite{AFJ,AJ}.

\begin{defi}
    For a graph $G$ and a vertex $i\in G$, \emph{the energy of the vertex}, which we denote by $\mathcal{E}(i)$, is given by

\[  \mathcal{E}(i)=|A(G)|_{ii}, \quad\quad~~~\text{for } i=1,\dots,n,
\]where $A(G)$ is the adjacency matrix of $G$.
\end{defi}

The energy of a vertex $\mathcal E(i)$ has been introduced in \cite{AJ} as a refinement of the energy of a graph. This concept has led to new bounds on the energy of a graph, see for example \cite{A-A,A-A2,AFJ,Fili,YLPL}.

One of the most important topological graph index is the so called Randi\'{c} index \cite{Ran}.
The Randi\'{c} index of $G$ is defined as

$$
R(G)
=
\sum_{(i,j)\in E(G)} \frac{1}{\sqrt{{deg(i)deg(j)}}}
$$

where $deg(i)$ denotes the degree of vertex $i$. In \cite{A-A}, an inequality between the graph energy and the Randi\'{c} index is given, namely 

$$\mathcal E(G)\geq 2R(G)$$

In this paper, we generalize this inequality by considering a weighted sum over the edges. Our inequality depends on a general choice of weights and one can give  specific weights to obtain particular inequalities.  The above inequality is given when one considers the weights to be $1/deg(i)$. Using this technique we prove that $$\mathcal E(G)\geq 2R(H)$$ for any subgraph $H$ (not necessarily induced). Other interesting old and new  inequalities for the energy of a graph are derived. These inequalities relate the energy with  different combinatorial quantities, such as the matching number \cite{Wong}, 
Hamiltonian paths and cycles \cite{Fili},
and $\{1,2\}$-factors \cite{A}.


\section{Main results}\label{mainsec}

The following inequality, proved in \cite{A-A}, is fundamental for results in this article. Since we will use the proof of the theorem we rewrite it here.  

\begin{thm}\cite{A-A}\label{T1}
Let $(i, j)$ be an edge of a simple (undirected) graph $G$. Then ${\mathcal E}(i){\mathcal E}(j)\geq1$. 
\end{thm}

\begin{proof}

 Let $A(G)$ be the adjacency matrix of $G$. Then, we can write $A(G)=UD U^t$, where $U=(u_{kl})$ is orthogonal and $D=(d_{kl})$ is a diagonal matrix such that $d_{kk}=\lambda_k$, with $\lambda_k$'s the eigenvalues of $A(G)$. A direct calculation shows that  ${\mathcal E}(i)=\sum_ku_{ik}^2|\lambda_k|$ and ${\mathcal E}(j)=\sum_ku_{jk}^2|\lambda_k|$. Moreover $A(G)_{ij}=\sum_ku_{ik}u_{jk}\lambda_k$. Since $i$ and $j$ are adjacent $A(G)_{ij}=1$.

Now consider $$v=(u_{i1}\sqrt{|\lambda_1|},\dots,u_{in}\sqrt{|\lambda_n|})$$ 
and
$$w=(u_{j1}sign(\lambda_1)\sqrt{|\lambda_1|},\dots,u_{jn}sign(\lambda_n)\sqrt{|\lambda_n|}),$$ 
then 
$$\left<v,w\right>^2=(\sum_ku_{ik}u_{jk}\lambda_k)^2=1,$$
$$||v||^2=\sum_ku_{ik}^2|\lambda_k|={\mathcal E}(i),$$
$$||w||^2=\sum_kv_{jk}^2|\lambda_k|={\mathcal E}(j).$$
which proves the assertion by the Cauchy-Schwarz inequality.
\end{proof}

\begin{lem}\label{lema}
Let $G$ be a simple and connected (undirected) graph. Then,  ${\mathcal E}(i){\mathcal E}(j)=1$ for all $(i,j)\in E(G)$  if and only if $G$ is a complete bipartite graph.   
\end{lem}
\begin{proof}

A direct calculation proves that if $G$ is a complete bipartite graph then  ${\mathcal E}(i){\mathcal E}(j)=1$ for all $(i,j)\in E(G)$, see Section 5.2 in \cite{AFJ}.

Conversely, suppose that ${\mathcal E}(i){\mathcal E}(j)=1$ for all edges.  We will prove that $A(G)$ has 2 non zero eigenvalues, for if $rank(A(G))=2$, according to Proposition 1 of \cite{cons}, $G$ is bipartite complete.

Suppose that $rank(A(G))\geq 3$ then $A(G)$ has at least 2 eigenvalues of the same sign. So suppose that $\lambda_r$ and $\lambda_s$ have the same sign, say $\sigma$. By the proof of Theorem \ref{T1}, if $(i, j)$  is an edge, and we let $v=(u_{i1}\sqrt{|\lambda_1|},\dots,u_{in}\sqrt{|\lambda_n|})$ 
and
$w=(u_{j1}sign(\lambda_1)\sqrt{|\lambda_1|},\dots,u_{jn}sign(\lambda_n)\sqrt{|\lambda_n|})$, then  $v=\kappa w$, since this is the condition for equality in Cauchy-Schwarz. 

Moreover, since $$||v||^2={\mathcal E}(i)$$
$$||w||^2={\mathcal E}(j),$$
then  $\kappa \neq 0$. In particular, this means that $$u_{ir}\sqrt{|\lambda_r|}=\kappa u_{jr}\sigma\sqrt{|\lambda_r|}$$
$$u_{is}\sqrt{|\lambda_s|}=\kappa u_{js}\sigma\sqrt{|\lambda_s|}$$

i.e.

$$u_{ir}=\sigma\kappa u_{jr},$$
$$u_{is}=\sigma\kappa u_{js}.$$

So for every $i$ in $V(G)$ consider a path  between $i$ and $1$, then by transitivity in the path we can write

$$u_{ir}=\kappa_{i} u_{1r},$$
$$u_{is}=\kappa_{i} u_{1s},$$
for some $\kappa_i\neq 0$. Now let $u_i=(u_{i1},\dots,u_{in})$, since $U$ is orthogonal then
$\{u_1,\dots, u_n\}$ is a basis for $\mathbb{R}^n$. This implies
$\{u_1, u_2/\kappa_2,\dots,u_n/\kappa_n\}$ is a basis. Now observe that  the $r$ and $s$ coordinates of $u_i/\kappa_i$  are exactly $u_{1r}$ and $u_{1s}$, which is contradiction, since  the elements of the basis of $\mathbb R^n$ cannot all have the same values in two of its coordinates.

\end{proof}

Now let $i$ be any vertex of $G$ and for any neighbor $j$ of $i$ define a weight $p_i^j$ such that   $0\leq p_i^j\leq1$ and \[\sum_{j\in N(i)}p_i^j=1\]
where $N(i)$ denote the set of neighbors of $i$, then we have the next important theorem:

\begin{thm}\label{Main}
Let $G$ be a graph with energy $\mathcal{E}(G)$ then 
\[\mathcal{E}(G)\geq 2 \sum_{(i,j)\in E(G)}\sqrt{p_i^jp_j^i}.\]
\end{thm}

\begin{proof}
Let $G=(V,E)$. 
Then, on one hand, \begin{eqnarray*}
\sum_{(i,j)\in E} p_i^j\mathcal{E}(i)+p_j^i\mathcal{E}(j)&=&\sum_{(i,j)\in E} p_i^j\mathcal{E}(i)+\sum_{(i,j)\in E}p_j^i\mathcal{E}(j)
\\&=&\frac{1}{2}\sum_{i\in V}\sum_{j\in N(i)} p_i^j\mathcal{E}(i)+ \frac{1}{2}\sum_{j\in V(G)}\sum_{i\in N(j)} p_j^i\mathcal{E}(j)\\
&=&\frac{1}{2}\sum_{i\in V} \mathcal{E}(i)+\frac{1}{2}\sum_{j\in V} \mathcal{E}(j)
=\mathcal{E}(G).
\end{eqnarray*}
On the other hand, by the classical AM-GM inequality, if $e=(i,j)$,

\begin{equation}\label{eql}
    p_i^j\mathcal{E}(i)+p_j^i\mathcal{E}(j)\geq2 \sqrt{p_i^j\mathcal{E}(i)p_j^i\mathcal{E}(j)}=2 \sqrt{p_i^jp_j^i}\sqrt{\mathcal{E}(i)\mathcal{E}(j)}\geq2 \sqrt{p_i^jp_j^i},
\end{equation}
where we used Theorem \ref{T1} in the last inequality. Finally, summing over all the edges of $G$ we obtain 

$$\mathcal{E}(G)=\sum_{(i,j)\in E} p_i^j\mathcal{E}(i)+p_j^i\mathcal{E}(j)\geq  \sum_{(i,j)\in E} 2 \sqrt{p_i^jp_j^i}=2  \sum_{(i,j)\in E}  \sqrt{p_i^jp_j^i}$$
as desired.

\end{proof}

\begin{prop}\label{prop:eq}
    Let $G$ be a graph with energy $\mathcal{E}(G)$, then there exist weights $\{p_i^j\}$ such that  
\[\mathcal{E}(G)= 2 \sum_{(i,j)\in E(G)}\sqrt{p_i^jp_j^i}.
\]    
    if and only if $G$ is the union of  complete bipartite graphs.
\end{prop}
\begin{proof}

For the equality to hold, Equation (\ref{eql}) implies that Lemma \ref{lema} should be satisfied, i.e. $G$ must be the union of complete bipartite graphs. On the other side if $G$ is the union of complete bipartite graphs then by taking
$p_i^j=1/deg(i)$ one gets $2R(G)$ on the right side. On the other hand, for the complete bipartite graph $K_{n,m}$, the spectrum is given by $\{\sqrt{nm},0 ^{(n+m-2)},-\sqrt{nm}\}$, see \cite{LSG}. Hence, in this case  $\mathcal E(G)=2\sqrt{nm}=2R(G)$.

\end{proof}

The following theorem, proved in \cite{A-A}, follows directly from Proposition \ref{prop:eq}. 

\begin{thm}
    \cite{A-A}\label{cota_randic}
   Let $G$ be a graph with  energy $\mathcal{E}(G)$ and Randi\'{c} index $R(G)$ then $\mathcal{E}(G)\geq 2R(G)$.  
\end{thm}
\begin{proof}
Taking $p_i^j=\frac{1}{deg(i)}$ we obtain in

$$\mathcal{E}(G)\geq  2 \sum_{(i,j)\in E(G)} \frac{1}{\sqrt{{deg(i)deg(j)}}}=2R(G)$$
\end{proof}

The next  generalization is one of the main results of the paper.
\begin{thm}\label{RH}
   Let $G$ be a graph with  energy $\mathcal{E}(G)$ and let $H$ be any subgraph of $G$  with Randi\'{c} index $R(H)$, then $\mathcal{E}(G)\geq 2R(H)$ with equality if and only if $G$ is the union of complete bipartite graphs and $R(H)=R(G)$.
\end{thm}

\begin{proof} It is easy to see that the proof of Theorem \ref{Main} follows exactly in the same way if one changes the second condition on the weights to \begin{equation}\label{weightin}
    \sum_{j\in N(i)}p_i^j\leq1.
\end{equation}
Now, for an arbitrary
subgraph $H$, and any edge $(i,j)\in E(H)$ let $p_i^j=\frac{1}{deg_H(i)}$ and for any edge $(i,j)\in E(G-H)$ let $p_i^j=0$, then (\ref{weightin}) is satisfied and we have
\begin{eqnarray*}
    \mathcal{E}(G)&\geq& 2  \sum_{(i,j)\in E(G)}  \sqrt{p_i^jp_j^i}\\&=&2\left( \sum_{(i,j)\in E(H)}  \sqrt{p_i^jp_j^i} +\sum_{(i,j)\in E(G-H)} \sqrt{p_i^jp_j^i}\right)\\&=&2\sum_{(i,j)\in E(H)}\frac{1}{\sqrt{{deg_H(i)deg_H(j)}}}\\&=&2R(H).
\end{eqnarray*}

One obtains the equality from  Proposition \ref{prop:eq} and its proof.

\end{proof}

The importance of the theorem above is that there may exist graphs $G$ and $H$ such that $H$ is subgraph of $G$ and $R(G)<R(H)$. In this way Theorem \ref{Main} improves Theorem \ref{cota_randic}. We now give an easy example where this happens.

\begin{exa} Let $G$ be the path of 4 vertices and $H$ two copies of $K_2$,  then
$H$ is a subgraph of  $G$ which is not induced. Observe that $2R(G)=1+2\sqrt{2}	\approx 3.82$,  $2R(H)=4$  and $\mathcal{E}(G)\approx 4.42$.

\end{exa}

In order to see the applications of Theorem \ref{RH} we prove some old and new results. We start with this classical inequality.

\begin{thm}
    \cite{CCGH} Let $G$ be a simple connected graph of size $n$. Then $\mathcal E(G)\geq 2\sqrt{n-1}$    
\end{thm}

 \begin{proof}
Let $H$ be any spanning tree of $G$. Then by Theorem \ref{RH} 

$$\mathcal E(G)\geq 2R(H).$$

On the other hand, among trees with $n$ vertices, the star $S_n$ has the minimum Randi\'c index, see \cite{LS}. Hence, since $$2R(H)\geq 2R(S_n)=2\sqrt{n-1}$$

the inequality holds by transitivity.

 \end{proof}

\begin{prop}
    
\label{correg}
    Let $G$ be a graph with a subgraph $H$ of size $k$ which is the disjoint union of regular graphs. Then  $\mathcal E(G)\geq k$, in particular if $k=n$ then G is non-hypoenergetic.   
\end{prop}

\begin{proof}
 Since  $H$ is the disjoint union regular graphs,
 then $R(H)=k/2$. By Theorem  \ref{RH} we obtain that 

   \[\mathcal{E}(G)\geq 2R(H)=k,\]

\end{proof}

The next known results follow from Proposition \ref{correg}, by the fact that $\{1,2\}$-factors, Hamiltonian cycles and matchings are union of regular graph.

\begin{thm}\cite{A} Let $G$ be a graph of order $n$. If G has a 
\{1,2\}-factor, then $\mathcal E(G) \geq n$.

\end{thm}

\begin{thm}\cite{Wong} \label{matching}
   Let $G$ be a graph with  energy $\mathcal{E}(G)$ and matching number $\mu(G)$, then $\mathcal{E}(G)\geq 2\mu(G)$.  
\end{thm}

\begin{thm}\cite{Fili}
    Let  $G$ be a graph of size n
    with a Hamiltonian cycle,
    then $\mathcal E(G) \geq n$. 
\end{thm}

Since the Randi\'c index of a path is given by $\frac{n-3+ 2 \sqrt{2}}{2}$ a similar result holds for graphs with Hamiltonian paths. Namely  if $G$ be a graph of size $n$
    with a Hamiltonian path,
    then $\mathcal E(G) \geq n-3 + 2 \sqrt{2}$. However, making use of Theorem \ref{Main} we can prove that  $\mathcal E(G)\geq n$, for $n$ even, and $\mathcal E(G)\geq \sqrt{n^2-1}$ for $n$ odd. This is stated in  Theorem \ref{HP} and proved in Section \ref{exsec}.

\begin{rem}
Observe that  AM-GM inequality implies that 

 \[2\sum_{(i,j)\in E}  \sqrt{p_i^jp_j^i}\leq  \sum_{(i,j)\in E}  (p_i^j+p_j^i)=\frac{1}{2}\sum_{i\in V}\sum_{j\in N(i)} p_i^j+ \frac{1}{2}\sum_{j\in V(G)}\sum_{i\in N(j)} p_j^i=n,\]
since $\sum_{j\in N(i)} p_i^j=1$. This is an small generalization of the fact that both the Randi\'c index and the matching number are less or equal than $n/2$. It also tell us that we can at most expect to bound by $n$. In this sense, for regular graphs our bound coincides with the Randi\'{c} index and for graphs with a perfect matching it also coincides with $n$.

\end{rem}

\section{Examples}\label{exsec}

    A natural question, in light of our results, is whether there  exist graphs such that  Theorem \ref{Main} improves Theorem \ref{RH}. The next two examples considered in \cite{AFJ} give a positive answer. Moreover, in the last example, asymptotically we have $\mathcal{E}(G)/2R(G)\approx 2/\sqrt{6}$, in contrast, for our new bound the quotient tends to 1. 

\subsection*{Path graph}

The \emph{path graph} $P_n$
is a graph with $n$ vertices
$\{v_1,...,v_{n}\}$ 
such that 
$v_{i}\sim v_{i+i}$
for $1\leq i \leq n-1$.

\begin{figure}[h!]
\centering

\begin{minipage}{0.45\textwidth}
\centering
\begin{tikzpicture}
\draw (0,0) -- (1,0);
\draw (1,0) -- (2,0);
\draw (2,0) -- (3,0);

\foreach \x/\label in {0/1,1/2,2/3,3/4}
    \filldraw (\x,0) circle (3pt) node[above] {};
\node [label=below: $P_4$] (*) at (1.5,-0.1) {};
\end{tikzpicture}
\end{minipage}
\begin{minipage}{0.45\textwidth}
\centering
\begin{tikzpicture}
\draw (0,0) -- (1,0);
\draw (1,0) -- (2,0);
\draw (2,0) -- (3,0);
\draw (3,0) -- (4,0);

\foreach \x/\label in {0/1,1/2,2/3,3/4,4/5}
    \filldraw (\x,0) circle (3pt) node[above]{};
\node [label=below: $P_5$] (*) at (2,-0.1) {};
\end{tikzpicture}
\end{minipage}
\caption{Path graphs $P_4$ and $P_5$}
\end{figure}
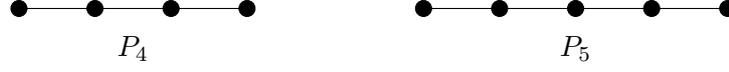

It is clear that if $n$ is even, 
then $P_n$ has a perfect matching
and consequently Collorary \ref{matching}
provides the bound
$\mathcal{E}(P_n) \geq n$,
which is the best possible
bound reached utilizing
Theorem \ref{Main}.

On the other hand, for odd $n$, consider the following weights

\[ p_{i}^{i+1} =  
\begin{cases}
\frac{n-i}{n-1}, & \text{if } i \text{ is odd} \\
\frac{i}{n+1}, & \text{if } i \text{ is even}
\end{cases}
\]
for $1\leq i \leq n-1$, and

\[ p_{i}^{i-1} =  
\begin{cases}
\frac{i-1}{n-1}, & \text{if } i \text{ is odd} \\
\frac{n+1-i}{n+1}, & \text{if } i \text{ is even}
\end{cases}
\]
for $2 \leq i \leq n$.
It is clear that
$ p_{i}^{i-1} +  p_{i}^{i+1} = 1$
for $2 \leq i \leq n-1$, and
$p_{1}^{2} = 1 = p_{n}^{n-1}$.
Moreover, note that
$p_{i}^{i+1} = p_{n+1-i}^{n-1}$
for $2 \leq i \leq n-1$
and then, 
the  weigths are symetrical
towards the center node $v_{\frac{n+1}{2}}$.

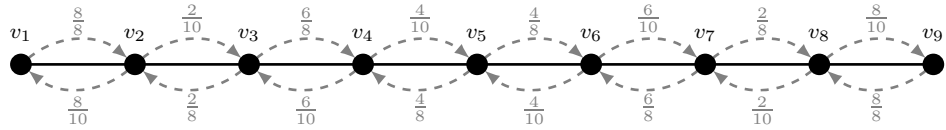
\begin{figure}[h!]
\centering
\begin{tikzpicture}
\pgfmathsetmacro{\nodeSize}{0.25}
\pgfmathsetmacro{\nodeDist}{1.5}
\pgfmathsetmacro{\nodeDistArriba}{0.5}
\pgfmathsetmacro{\grosorAristas}{1}
\pgfmathsetmacro{\grosorFlechas}{1}
\pgfmathsetmacro{\anguloFlechas}{40}
\pgfmathsetmacro{\colorFlechas}{"gray"}
\Vertex[size = \nodeSize, color=black,
x = \nodeDist*0, label=$v_1$, position=above, 
fontscale = 1, distance = \nodeDistArriba mm]{1}
\Vertex[size = \nodeSize, color=black,
x = \nodeDist*1, label=$v_2$, position=above, 
fontscale = 1, distance = \nodeDistArriba mm]{2}
\Vertex[size = \nodeSize, color=black,
x = \nodeDist*2, label=$v_3$, position=above,
fontscale = 1, distance = \nodeDistArriba mm]{3}
\Vertex[size = \nodeSize, color=black,
x = \nodeDist*3, label=$v_4$, position=above, 
fontscale = 1, distance = \nodeDistArriba mm]{4}
\Vertex[size = \nodeSize, color=black,
x = \nodeDist*4, label=$v_5$, position=above, 
fontscale = 1, distance = \nodeDistArriba mm]{5}
\Vertex[size = \nodeSize, color=black,
x = \nodeDist*5, label=$v_6$, position=above, 
fontscale = 1, distance = \nodeDistArriba mm]{6}
\Vertex[size = \nodeSize, color=black,
x = \nodeDist*6, label=$v_7$, position=above, 
fontscale = 1, distance = \nodeDistArriba mm]{7}
\Vertex[size = \nodeSize, color=black,
x = \nodeDist*7, label=$v_8$, position=above, 
fontscale = 1, distance = \nodeDistArriba mm]{8}
\Vertex[size = \nodeSize, color=black,
x = \nodeDist*8, label=$v_9$, position=above, 
fontscale = 1, distance = \nodeDistArriba mm]{9}
\Edge[lw= \grosorAristas pt, color=black](1)(2)
\Edge[lw= \grosorAristas pt, color=black](2)(3)
\Edge[lw= \grosorAristas pt, color=black](3)(4)
\Edge[lw= \grosorAristas pt, color=black](4)(5)
\Edge[lw= \grosorAristas pt, color=black](5)(6)
\Edge[lw= \grosorAristas pt, color=black](6)(7)
\Edge[lw= \grosorAristas pt, color=black](7)(8)
\Edge[lw= \grosorAristas pt, color=black](8)(9)
\Edge[lw= \grosorFlechas pt, color= \colorFlechas,
bend= \anguloFlechas, label= {$\frac{8}{8}$},
position=above, style={dashed}, Direct](1)(2)
\Edge[lw= \grosorFlechas pt, color= \colorFlechas,
bend= \anguloFlechas, label=$\frac{2}{10}$,
position=above, style={dashed}, Direct](2)(3)
\Edge[lw= \grosorFlechas pt, color=\colorFlechas,
bend= \anguloFlechas, label= $\frac{6}{8}$,
position=above, style={dashed}, Direct](3)(4)
\Edge[lw= \grosorFlechas pt, color=\colorFlechas,
bend= \anguloFlechas, label=$\frac{4}{10}$,
position=above, style={dashed}, Direct](4)(5)
\Edge[lw= \grosorFlechas pt, color=\colorFlechas,
bend= \anguloFlechas, label= $\frac{4}{8}$,
position=above, style={dashed}, Direct](5)(6)
\Edge[lw= \grosorFlechas pt, color=\colorFlechas,
bend= \anguloFlechas, label= $\frac{6}{10}$,
position=above, style={dashed}, Direct](6)(7)
\Edge[lw= \grosorFlechas pt, color=\colorFlechas,
bend= \anguloFlechas, label= $\frac{2}{8}$,
position=above, style={dashed}, Direct](7)(8)
\Edge[lw= \grosorFlechas pt, color=\colorFlechas,
bend= \anguloFlechas, label= $\frac{8}{10}$,
position=above, style={dashed}, Direct](8)(9)
\Edge[lw= \grosorFlechas pt, color=\colorFlechas,
bend=  \anguloFlechas, label=$\frac{8}{10}$,
position=below, style={dashed}, Direct](2)(1)
\Edge[lw= \grosorFlechas pt, color=\colorFlechas,
bend=  \anguloFlechas, label=$\frac{2}{8}$,
position=below, style={dashed}, Direct](3)(2)
\Edge[lw= \grosorFlechas pt, color=\colorFlechas,
bend=  \anguloFlechas, label=$\frac{6}{10}$,
position=below, style={dashed}, Direct](4)(3)
\Edge[lw= \grosorFlechas pt, color=\colorFlechas,
bend=  \anguloFlechas, label=$\frac{4}{8}$,
position=below, style={dashed}, Direct](5)(4)
\Edge[lw= \grosorFlechas pt, color=\colorFlechas,
bend= \anguloFlechas, label=$\frac{4}{10}$,
position=below, style={dashed}, Direct](6)(5)
\Edge[lw= \grosorFlechas pt, color=\colorFlechas,
bend= \anguloFlechas, label=$\frac{6}{8}$,
position=below, style={dashed}, Direct](7)(6)
\Edge[lw= \grosorFlechas pt, color=\colorFlechas,
bend= \anguloFlechas, label=$\frac{2}{10}$,
position=below, style={dashed}, Direct](8)(7)
\Edge[lw= \grosorFlechas pt, color=\colorFlechas,
bend= \anguloFlechas, label=$\frac{8}{8}$,
position=below, style={dashed}, Direct](9)(8)
\end{tikzpicture}
\caption{Weights example for $P_9$}
\end{figure}

Since $n$ is odd, it can
be written as $n = 2k + 1$, for
$k \in \mathbb{N}$, 
then $n-1 = 2k$
and we can compute.

\begin{eqnarray*}
2\sum_{(i,j)\in E(G)}\sqrt{p_i^jp_j^i} &=&
2  \sum_{i = 1}^{2k}
\sqrt{p_{i}^{i+1} p_{i+1}^{i}}\\
&=& 2 \left(
\sum_{i = 1}^{k}
\sqrt{p_{2i}^{2i+1} p_{2i+1}^{2i}} 
+ \sum_{i = 1}^{k}
\sqrt{p_{2i - 1}^{2i} p_{2i}^{2i - 1}} 
\right) \\
&=& 2 \left(
\sum_{i = 1}^{k}
\sqrt{\frac{2i}{n+1} \frac{2i}{n-1}} 
+ \sum_{i = 1}^{k}
\sqrt{\frac{n-2i+1}{n-1} \frac{n-2i+1}{n+1}} 
\right) \\
&=& \frac{2}{\sqrt{(n-1)(n+1)}} \left(
\sum_{i = 1}^{k} 2i
+
\sum_{i = 1}^{k} n - 2i + 1
\right) \\
&=&
\frac{(2k)(n+1)}{\sqrt{(n-1)(n+1)}} \\
&=&
\frac{(n-1)(n+1)}{\sqrt{(n-1)(n+1)}}\\
&=&
\sqrt{n^2 - 1}
\end{eqnarray*}

Therefore, using 
Theorem \ref{Main}
we get the bound
$\mathcal{E}(P_n) \geq \sqrt{n^2 - 1}$
for odd $n$.
Furthermore, this 
allows us to bound the energy of graphs
having Hamiltonian paths.
Namely, we get the following 
theorem.

\begin{thm}
\label{HP}
    Let  $G$ be a graph of size $n$
    with a Hamiltonian path.
    If $n$ is even then $\mathcal E(G) \geq n$. 
   If $n$ is odd then
    $\mathcal E(G) \geq \sqrt{n^2 - 1}$.
\end{thm}

\subsection*{Dandelion graphs}

 We call a \emph{dandelion}, $D_n$, 
a graph with $3n+1$ vertices
$\{v_1,...,v_{3n+1}\}$ in which
$v_1$ is connected to every vertex 
in the set $\{v_2,\dots,v_n+1\}$ and,
for $2\leq i \leq n+1$, 
$v_{i}\sim v_{n+i}$ and $v_{i}\sim v_{2n+i}$ .
This is 
not an edge transitive graph, 
but there are three types of vertices:
(1) $v_1$, (2) neighbours of $v_1$ 
and (3) leafs.

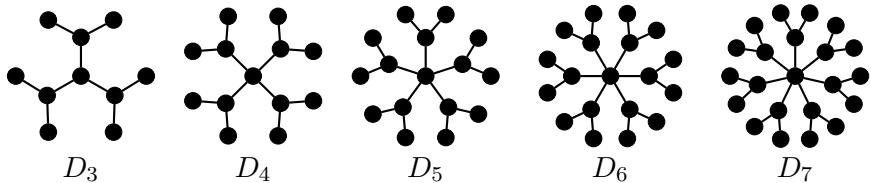
\begin{figure}[h]
\begin{tikzpicture}
[mystyle/.style={scale=0.6, draw,shape=circle,fill=black}]
\def\ngon{3}
\def\nngon{6}
\node[regular polygon,regular polygon sides=\ngon,minimum size=1cm] (p) {};
\foreach\x in {1,...,\ngon}{\node[mystyle] (p\x) at (p.corner \x){};}
\node[mystyle] (p0) at (0,0) {};
\foreach\x in {1,...,\ngon}
{
 \draw[thick] (p0) -- (p\x);
}

\node[regular polygon,regular polygon sides=2*\ngon,minimum size=1.7cm] (r) {};
\foreach\x in {1,...,\nngon}{\node[mystyle] (r\x) at (r.corner  \x ){};}
\foreach\x in {1,...,\ngon}{
\foreach\y in {\the\numexpr2*\x-1\relax,...,\the\numexpr2*\x\relax}
    \draw[thick] (p\x) -- (r\y);
  }

  \node [label=below:$D_{\ngon}$] (*) at (0,-0.8) {};
 \end{tikzpicture}\quad
\begin{tikzpicture}
[mystyle/.style={scale=0.6, draw,shape=circle,fill=black}]
\def\ngon{4}
\def\nngon{8}
\node[regular polygon,regular polygon sides=\ngon,minimum size=1cm] (p) {};
\foreach\x in {1,...,\ngon}{\node[mystyle] (p\x) at (p.corner \x){};}
\node[mystyle] (p0) at (0,0) {};
\foreach\x in {1,...,\ngon}
{
 \draw[thick] (p0) -- (p\x);
}

\node[regular polygon,regular polygon sides=2*\ngon,minimum size=1.7cm] (r) {};
\foreach\x in {1,...,\nngon}{\node[mystyle] (r\x) at (r.corner  \x ){};}
\foreach\x in {2,...,\ngon}{
\foreach\y in {\the\numexpr2*\x-2\relax,...,\the\numexpr2*\x-1\relax}
    \draw[thick] (p\x) -- (r\y);
  }

  \draw[thick] (p1) -- (r1);

 \draw[thick] (p1) -- (r\nngon);
  \node [label=below:$D_{\ngon}$] (*) at (0,-0.8) {};
 \end{tikzpicture}\quad
\begin{tikzpicture}
[mystyle/.style={scale=0.6, draw,shape=circle,fill=black}]
\def\ngon{5}
\def\nngon{10}
\node[regular polygon,regular polygon sides=\ngon,minimum size=1cm] (p) {};
\foreach\x in {1,...,\ngon}{\node[mystyle] (p\x) at (p.corner \x){};}
\node[mystyle] (p0) at (0,0) {};
\foreach\x in {1,...,\ngon}
{
 \draw[thick] (p0) -- (p\x);
}

\node[regular polygon,regular polygon sides=2*\ngon,minimum size=1.7cm] (r) {};
\foreach\x in {1,...,\nngon}{\node[mystyle] (r\x) at (r.corner  \x ){};}
\foreach\x in {1,...,\ngon}{
\foreach\y in {\the\numexpr2*\x-1\relax,...,\the\numexpr2*\x\relax}
    \draw[thick] (p\x) -- (r\y);
  }
  \node [label=below:$D_{\ngon}$] (*) at (0,-0.8) {};
 \end{tikzpicture}
  \quad
\begin{tikzpicture}
[mystyle/.style={scale=0.6, draw,shape=circle,fill=black}]
\def\ngon{6}
\def\nngon{12}\node[regular polygon,regular polygon sides=\ngon,minimum size=1cm] (p) {};
\foreach\x in {1,...,\ngon}{\node[mystyle] (p\x) at (p.corner \x){};}
\node[mystyle] (p0) at (0,0) {};
\foreach\x in {1,...,\ngon}
{
 \draw[thick] (p0) -- (p\x);
}

\node[regular polygon,regular polygon sides=2*\ngon,minimum size=1.7cm] (r) {};
\foreach\x in {1,...,\nngon}{\node[mystyle] (r\x) at (r.corner  \x ){};}
\foreach\x in {2,...,\ngon}{
\foreach\y in {\the\numexpr2*\x-2\relax,...,\the\numexpr2*\x-1\relax}
    \draw[thick] (p\x) -- (r\y);
  }

  \draw[thick] (p1) -- (r1);

 \draw[thick] (p1) -- (r\nngon);
  \node [label=below:$D_{\ngon}$] (*) at (0,-0.8) {};
 \end{tikzpicture}
  \quad\begin{tikzpicture}
[mystyle/.style={scale=0.6, draw,shape=circle,fill=black}]
\def\ngon{7}
\def\nngon{14}
\node[regular polygon,regular polygon sides=\ngon,minimum size=1cm] (p) {};
\foreach\x in {1,...,\ngon}{\node[mystyle] (p\x) at (p.corner \x){};}
\node[mystyle] (p0) at (0,0) {};
\foreach\x in {1,...,\ngon}
{
 \draw[thick] (p0) -- (p\x);
}

\node[regular polygon,regular polygon sides=2*\ngon,minimum size=1.7cm] (r) {};
\foreach\x in {1,...,\nngon}{\node[mystyle] (r\x) at (r.corner  \x ){};}
\foreach\x in {1,...,\ngon}{
\foreach\y in {\the\numexpr2*\x-1\relax,...,\the\numexpr2*\x\relax}
    \draw[thick] (p\x) -- (r\y);
  }
  \node [label=below:$D_{\ngon}$] (*) at (0,-0.8) {};
 \end{tikzpicture}

 \caption{Dandelion graphs $D_3$, $D_4$, $D_5$, $D_6$ and $D_7$. }
\end{figure}

The spectra of $D_n$ is
$ \{ 
[-\sqrt{n+2}] ^{1} ,
[-\sqrt{2}] ^{n-1} ,
[0] ^{n+1} ,
[\sqrt{2}] ^ {n-1} ,
[\sqrt{n+2}] ^{1} 
\}$, see \cite{AFJ}.
The total energy  is given by 
$\mathcal{E}(D_n)=2(n-1)\sqrt{2}+ 2\sqrt{n+2}.$


The Randi\'{c} index is easily computed,

\begin{eqnarray*}
    R(D_n) 
            &=& \sum_{i = 2}^{n+1}
                \frac{1}{\sqrt{{deg(1)deg(v_i)}}}
                +
                \frac{1}{\sqrt{{deg(v_i)deg(v_{n+i})}}}
                +
                \frac{1}{\sqrt{{deg(v_i)deg(v_{2n+i})}}} \\
            &=& n \frac{1}{\sqrt{3n}} + 
               2n \frac{v_1}{\sqrt{3}}\\
           &=& \frac{\sqrt{n} + 2n}{\sqrt{3}}.
\end{eqnarray*}

Therefore, Corollary \ref{cota_randic}
provides the bound
$\mathcal{E}(G)\geq \frac{4n + 2 \sqrt{n}}{\sqrt{3}}$. On the other hand, 
for $n \geq 2$,
consider the subgraph $H$
formed by all the nodes in $D_n$,
and the connections
$v_{i}\sim v_{n+i}$ and $v_{i}\sim v_{2n+i}$
for $2\leq i \leq n+1$,
and $v_{1}\sim v_{2}$.
Note that $H$ is the disjoint union
of $n-1$ copies of $P_3$ and one copy of $S_4$.
The Randi\'{c} index of $H$ is given by

\begin{eqnarray*}
    R(H) &=&  (n-1) R(P_3) + R(S_4) \\
         &=& (n-1) \sqrt{2} + \sqrt{3}.
\end{eqnarray*}

Therefore, by utilizing Theorem \ref{RH}
we get that
$\mathcal{E}(G)\geq 2\sqrt{2} (n-1)+ 2 \sqrt{3}$. It is not hard to prove that this subgraph is maximal with respect to the Randi\'{c} index. Furthermore,
consider the weights
$p_{i}^{1} = \frac{1}{2n+1}$,
$p_{1}^{i} = \frac{1}{n}$,
$p_{i}^{n+i} = p_{i}^{2n+i} = \frac{n}{2n+1}$
and
$p_{n+i}^{i} = p_{2n+i}^{i} = 1$
for $2\leq i \leq n+1$.
Then, by using Theorem
\ref{Main} we get that

\begin{eqnarray*}
    \mathcal{E}(G) &\geq& 
    2 \sum_{(i,j)\in E(G)} \sqrt{p_i^j p_j^i}\\
    &=& 2  \sum_{i = 2}^{n+1}
    \sqrt{p_{i}^{1} p_{1}^{i}} + 
    \sqrt{p_{i}^{n+i} p_{n+i}^{i}} + 
    \sqrt{p_{i}^{2n+i} p_{2n+i}^{i}} \\
    &=&
    2n \sqrt{\frac{1}{2n+1} \frac{1}{n}}
    + 4n \sqrt{\frac{n}{2n+1}} \\
    &=&
    \frac{2\sqrt{n}}{\sqrt{2n+1}} +
    \frac{4 n \sqrt{n}}{\sqrt{2n+1}} \\
    &=&
    \frac{2\sqrt{n}(1+2n)}{\sqrt{2n+1}}\\
    &=&
    2\sqrt{n}\sqrt{2n+1}\\
    &=&
    2 \sqrt{2n^2 + n}. 
\end{eqnarray*}

Theorem
\ref{Main} yields the bound
$\mathcal{E}(G) \geq 2 \sqrt{2n^2 + n}$,
which improves the bounds
provided by 
Corollary \ref{cota_randic}
and Theorem \ref{RH}.

\begin{prob}
For a given graph $G$, find the weights with maximize  \[\sum_{(i,j)\in E(G)}  \sqrt{p_i^jp_j^i}.\]
\end{prob}

\begin{prob}
For a given graph $G$, find a subgraph $H$ which is maximal with respect to the Randi\'c index.    
\end{prob}

Using a computer program, we have identified the types of maximal subgraphs in terms of the Randi\'c index for graphs with up to 7 vertices. In every case, these subgraphs turn out to be unions of regular and bipartite graphs.

Solving both problems
would lead to better bounds
by means of Theorems
\ref{Main} and \ref{RH}, respectively.

\noindent Department of Actuarial Sciences, Physics and Mathematics. Universidad de las Am\'{e}ricas Puebla. San Andr\'{e}s Cholula, Puebla. M\'{e}xico.\\\noindent Email: \emph{gerardo.arizmendi@udlap.mx, diego.huertaoa@udlap.mx}
	\\~
	

\end{document}